\documentclass[12pt]{amsart}
\usepackage{amsmath,amssymb,amsthm,amscd,amsfonts}
\usepackage{hyperref, color} 
\usepackage{comment} 
\usepackage{tikz}
\usepackage{tikz-cd}
\usepackage{graphicx}

\usepackage{biblatex}
\begin{document}

\newtheorem{thm}{Theorem}[section]
\newtheorem{lem}[thm]{Lemma}
\newtheorem{defi}[thm]{Definition}
\newtheorem{prop}[thm]{Proposition}
\newtheorem{claim}[thm]{Claim}

\theoremstyle{definition}
\newtheorem{remark}[thm]{Remark}

\def\R{\mathbb{R}}
\def\E{\mathbb{E}}
\def\C{\mathbb{C}}
\def\N{\mathbb{N}}
\def\Z{\mathbb{Z}}
\def\W{\Omega}
\def\d{\delta}
\def\a{\alpha}
\def\e{\varepsilon}
\def\t{\tau}
\def\x{\xi}
\def\r{\rho}
\def\n{\nu}
\def\th{\theta}
\def\m{\mu}
\def\s{\sigma}
\def\o{\omega}
\def\l{\lambda}
\def\k{\kappa}

\author{Theodora Bourni}
\address{Department of Mathematics, University of Tennessee Knoxville, Knoxville TN, 37996}
\email{tbourni@utk.edu}
\author{Nathan Burns}
\address{Department of Mathematics, University of Tennessee Knoxville, Knoxville TN, 37996}
\email{nburns6@vols.utk.edu}
\author{Spencer Catron}
\address{Department of Mathematics, University of Washington, Seattle WA, 98195}
\email{scatron@uw.edu}

\title[Classification of Convex, Ancient, Free Boundary CSF]{Classification of Convex Ancient Solutions to  Free Boundary Curve Shortening Flow in Convex Domains.}

\begin{abstract}
We classify convex ancient curve shortening flows with free boundary on general bounded convex domains. 
\end{abstract}
\maketitle

\tableofcontents

\section{Introduction}

A smooth one-parameter family of immersions $X:M_{1} \times [0,T) \to \R^{2}$ is said to evolve via curve shortening flow if it satisfies
\begin{equation}\label{eq: csf}
	\partial_{t}X(u,t) = \vec{\k}
\end{equation}
where $\vec{\k}(u,t)$ is the curvature vector of $\Gamma_{t} := X(M_{1},t)$. It was shown by Gage and Hamilton \cite{hamilton} that, given any compact, convex, embedded initial condition $\Gamma_{0}$, there exists a solution to \eqref{eq: csf}, which exists on a maximal time interval, and further, Grayson extended this result for an arbitrary embedded closed curve. A solution to curve shortening flow is called \emph{ancient} if it exists on some time interval which contains an interval of the form $(\infty,a]$ for some $a < \infty$, which by time translation we can take to be $0$. Examples of ancient solutions include the stationary line, the shrinking circle, the Angenent oval and the grim reaper. 
The free boundary curve shortening flow is the following Neumann problem:
\begin{equation}\label{eq:fbcsf}
	\begin{cases}
		\partial_{t}X(u,t) = \vec{\k} \:\:\:\:\: \text{on} \:\: \overset{\circ}{M_{1}} \\
		X_{t}(\partial M_{1}) \subset \partial\Omega \\
		\langle \nu, \nu^{\Omega} \circ X \rangle = 0 \:\: \text{on} \:\: \partial M_{1},
	\end{cases}
\end{equation}
where $\Omega \subset \R^{2}$ is some closed and connected domain with boundary $\partial\Omega$. That is, the endpoints of the family $\Gamma_{t}$ remain orthogonal to a fixed supporting curve $\partial\Omega$. It has been shown that convex curves which lie inside a convex domain, with free boundary with respect to $\partial\Omega$, remain convex with respect to the free boundary curve shortening flow and moreover, shrink to a point on the boundary curve \cite{convergence, regularity}. Very recently, Langford and Zhu \cite{langford2023distance} proved a Grayson-type  theorem in the free boundary case, by proving that embedded curves converge in infinite time to a (unique) ``critical chord'', or contract in finite time to a ``round half-point'' on $\partial \Omega$. The classification of convex ancient solutions to the free boundary curve shortening flow was initiated by Bourni and Langford \cite{classification}, who proved that, up to rotation, there is a unique solution to this problem on the disk. In this paper we will adapt those methods to produce the following analogous result on strictly convex domains of $\R^{2}$:
\begin{thm}
	Modulo time translation, for each diameter of $\Omega$, there exist precisely two convex, locally uniformly convex, ancient solutions to the free boundary curve shortening flow in $\Omega$, one lying on each side of the diameter. By a diameter, we mean a line segment intersecting $\partial\Omega$ orthogonally.
\end{thm}
\addtocontents{toc}{\setcounter{tocdepth}{0}}
\section*{Acknowledgements}
\addtocontents{toc}{\setcounter{tocdepth}{2}}
The first and second named authors were supported by grant DMS-2105026 of the National Science Foundation.

The authors would like to thank Mat Langford for useful conversations on the topic.

\section{Existence}

In this section we will provide an explicit construction of a non-trivial ancient solution to the free boundary curve shortening flow emanating from a diameter of a convex domain. In the ensuing discussion, we shall let $\Omega$ be a strictly convex domain in $\R^{2}$ and consider a diameter of $\Omega$, $D$, that is, a line segment which intersects $\partial\Omega$ orthogonally. Note that such a diameter always exists; consider the segment which maximises the Euclidean distance amongst pairs of points on $\partial\Omega$. By scaling, translating and rotating, we may and henceforward will assume that $D$ lies on the $x$-axis with endpoints $\pm e_{1}$. 

\subsection{Barriers}

We parameterise $\partial\Omega$ via the turning angle
\begin{equation}\label{eq:boundary paramatrisation}
	\Phi:[-\tfrac{\pi}{2},\tfrac{3\pi}{2}] \to \R^{2}
\end{equation}
so that the tangent and outward unit normal to $\partial\Omega$ at $\Phi(\omega)$ are given by $\tau^{\Omega}(\omega) = (\cos\omega,\sin\omega)$ and $\nu^{\Omega}(\omega) = (\sin\omega,-\cos\omega)$. Notice that with this parameterisation and under our assumption on the diameter $D$, $\Phi(\pm \frac{\pi}{2}) = \pm e_{1}$.
Define $C^{\pm}_{r}$ to be two circles of radius $r$ that lie inside $\Omega$ and are tangent to $\partial\Omega$ at the points $\pm e_{1}$ respectively, which we can always do since the boundary is smooth (see Figure \ref{fig: barrier circles}). Moreover, we choose $r$ to satisfy
\begin{equation}\label{eq:assumptions on r}
	4r \leq \k^{\Omega}(\omega) \leq \frac{1}{4r}
\end{equation}
for all $\omega$, the reason for which will be made apparent later.

\begin{defi}\label{definitionofacuteangle}
	Let $\Gamma$ be a convex curve intersecting a circle $C$ at a point $p$, and consider the radial segment $R$ passing through $p$ and the centre of $C$. We shall say that $\Gamma$ intersects $C$ at \emph{an acute angle at $p$}, if the tangent to $\Gamma$ at $p$ locally separates $\Gamma$ and $R$ inside $C$. 
\end{defi}

Next we show that convex curves in $\Omega$ intersecting the boundary orthogonally near $D$ with $y > 0$ intersect $C^{\pm}_{r}$ in acute angles. This will enable us to create upper barriers for a solution of \eqref{eq:fbcsf}.

\begin{lem}\label{acuteangle}
	Let $\Gamma$ be a convex curve inside $\Omega$ which intersects the boundary orthogonally at two points in $\{r > y > 0\}$. Then $\Gamma$ intersects $C^{-}_{r}$ (resp. $C^{+}_{r}$) transversally at two points, and at the intersection point $p$ with the smallest (resp. larger) $x$-coordinate, $\Gamma$ intersects $C^{-}_{r}$ (resp. $C^{+}_{r}$) at an acute angle. 
\end{lem}
\begin{proof} We prove the lemma for $C^{-}_{r}$, as the proof for $C^{+}_{r}$ is similar.
Consider a parametrisation of $C^{-}_{r}$ by turning angle
\[
C:[-\tfrac\pi2, \tfrac{3\pi}{2}]\to \R^2\,.
\]
For any $\theta_0\in (\pi, \tfrac{3\pi}{2})$, using the fact that $C(\tfrac {3\pi}{2})=\Phi(\tfrac {3\pi}{2})=-e_1$, we have
\begin{equation}\label{hor}
\begin{split}
    \langle e_1, \Phi(\theta_0)\rangle-\langle e_1, C(\th_0)\rangle&=\int_{\tfrac{3\pi}{2}}^{\theta_0} \langle e_1,\Phi'(\theta)-C'(\theta)\rangle d\th\\
    &= \int_{\tfrac{3\pi}{2}}^{\theta_0}\cos\theta\left(\frac{1}{\k(\th)}-r\right)d\th \ge 0\,,
\end{split}
\end{equation}
with the last inequality being true because of \eqref{eq:assumptions on r}.

Since $\Gamma$ is convex, we can parametrise it by turning angle, $\gamma := \gamma(\theta)$, and since $\Gamma\cap \partial\Omega\subset\{y>0\}$, we can assume that there exists a point $x \in \Gamma$ so that $x = \gamma(0)$. Let $p_0=\gamma(\omega_0)$ and $p_1=\gamma(\omega_1)$ be the points of intersection of $\Gamma$ with $\partial \Omega$ and $C^{-}_{r}$ respectively with the smaller $x$-coordinates. Denoting by $\theta_1\in (\frac\pi2, \frac{3\pi}{2})$ the angle for which $p_1=C(\theta_1)$, the lemma is equivalent to showing that
\begin{equation}\label{anglesclaim}
(\cos\theta_1, \sin\theta_1)\cdot (\cos\omega_1, \sin\omega_1)<0\,,
\end{equation}
that is, the tangent vectors of $\Gamma$ and $C^{-}_{r}$ at $p_1$ form an angle that is bigger than $\frac\pi2$.

Let $\theta_0\in (\pi, \frac{3\pi}{2})$ be such that $\Phi(\theta_0)=p_0$. Then 
\[
(\cos\omega_0, \sin\omega_0)=(\sin\theta_0, -\cos\theta_0)\,,
\]
and therefore, $\theta_0=\omega_0+\frac{3\pi}{2}$.
Consider the tangent to $\Gamma$ at $p_0$ and let $p_2$ be the point of intersection of this tangent and $C^{-}_{r}$ with the smaller $x$-coordinate. Let $\theta_2\in (\frac\pi2, \frac{3\pi}{2})$ be such that $p_2=C(\theta_2)$. Then, using the convexity of $\Gamma$ we have that 
\[
\langle p_0, e_1\rangle<\langle p_2, e_1\rangle<\langle p_1, e_1\rangle
\]
and thus, using \eqref{hor}, we obtain $\theta_0>\theta_2>\theta_1$. Now, since $\theta_1>\theta_0$, by convexity of $\Gamma$, we have
\[
\omega_1+\tfrac {3\pi}{2}>\theta_1\,.
\]
Recalling the domains of definition for $\o_1$ and $\th_1$, this implies \eqref{anglesclaim}.
\end{proof}

\begin{remark}\label{acuteangleremark}
It should be noted here (as it will be useful later) that the proof of Lemma~\ref{acuteangle} does not require that the curvature of $C_r^-$ is constant but merely that the minimum curvature of $C_r^-$ is bigger than the maximum of $\partial \Omega$ around the point of contact. 
\end{remark}

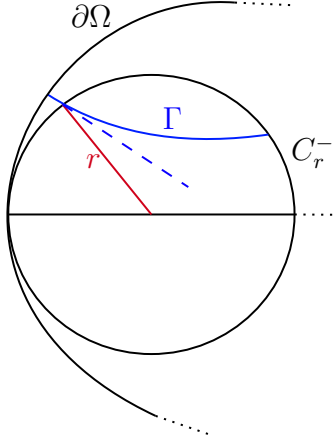
\begin{figure}[t]
\begin{center}

\tikzset{every picture/.style={line width=0.75pt}} 

\begin{tikzpicture}[x=0.75pt,y=0.75pt,yscale=-1,xscale=1,trim left=2cm]

\draw   (64.27,167.15) .. controls (64.27,128.25) and (96.59,96.71) .. (136.45,96.71) .. controls (176.32,96.71) and (208.64,128.25) .. (208.64,167.15) .. controls (208.64,206.06) and (176.32,237.6) .. (136.45,237.6) .. controls (96.59,237.6) and (64.27,206.06) .. (64.27,167.15) -- cycle ;
\draw    (64.27,167.15) -- (208.64,167.15) ;
\draw [color={rgb, 255:red, 208; green, 2; blue, 27 }  ,draw opacity=1 ]   (91.59,111.56) -- (136.45,167.15) ;
\draw [color={rgb, 255:red, 0; green, 32; blue, 255 }  ,draw opacity=1 ]   (84.1,106.56) .. controls (119.66,129.54) and (158.84,131.95) .. (195.45,126.81) ;
\draw [color={rgb, 255:red, 11; green, 0; blue, 255 }  ,draw opacity=1 ] [dash pattern={on 4.5pt off 4.5pt}]  (91.59,111.56) -- (155.06,153.31) ;
\draw    (138.57,268.81) .. controls (2.13,204.81) and (74.91,52.91) .. (178.52,60.19) ;
\draw  [dash pattern={on 0.84pt off 2.51pt}]  (178.52,60.19) -- (206.61,61.48) ;
\draw  [dash pattern={on 0.84pt off 2.51pt}]  (138.57,268.81) -- (167.1,278.67) ;
\draw  [dash pattern={on 0.84pt off 2.51pt}]  (208.64,167.15) -- (229,167.29) ;

\draw (102,135.33) node [anchor=north west][inner sep=0.75pt]    {$\textcolor[rgb]{0.82,0.01,0.11}{r}$};
\draw (140.64,111.46) node [anchor=north west][inner sep=0.75pt]  [color={rgb, 255:red, 6; green, 0; blue, 255 }  ,opacity=1 ]  {$\Gamma $};
\draw (94.49,59.76) node [anchor=north west][inner sep=0.75pt]    {$\partial \Omega $};
\draw (205.43,125.97) node [anchor=north west][inner sep=0.75pt]    {$C_{r}^{-}$};

\end{tikzpicture}
\end{center}
\caption{$\Gamma$ intersects $C^{-}_{r}$ at an Acute Angle.}
\label{fig: acute angle}
\end{figure}

With all this in mind, we are ready to construct upper barriers for solutions to the free boundary curve shortening flow in $\Omega$.
Consider the arcs
\[
	K_{\omega} := \{(x,y) \in B^{1}: x^{2} + (\csc\omega - y)^{2} = \cot^{2}\omega\}, \omega \in (0,\tfrac{\pi}{2})
\]
which intersect $S^{1}$ orthogonally at the points $(\pm \cos\omega,\sin\omega)$. In \cite{classification}, it was observed that if we set $\omega(t) = \arcsin e^{2t}$, $t \in (-\infty,0)$, then the inward normal speed of $K_{\omega(t)}$ is no less than its curvature. By scaling and translating we see that the same is true for 
\[
	K_{t}^{\pm} := rK_{\omega(r^{-2}t)} \pm (1-r)e_{1},
\]
which are now families of curves that intersect $C^{\pm}_{r}$ orthogonally. Define the curves
\begin{equation}\label{barriercurve}
	K_{t} := K^{+}_{t} \cup L_{t} \cup K_{t}^{-}
\end{equation}
where $L_{t}$ is the line joining the point $K_{t}^{-} \cap C^{-}_{r}$ with the larger $x$-coordinate to the point of $K^{+}_{t} \cap C^{+}_{r}$ with the smaller $x$-coordinate (see Figure \ref{fig: barrier circles}). The maximum principle along with Lemma \ref{acuteangle} gives us the following:

\begin{prop}\label{barrier}
	A convex solution to the free boundary curve shortening flow in $\Omega$ which lies below $K_{\hat{t}}$ and a time $t_{0}$ for some $\hat{t} \in (-\infty,0)$ must lie below $K_{\hat{t} +t -t_{0}}$ for all $t > t_{0}$ such that $t - t_{0} + \hat{t} < 0$.
\end{prop}

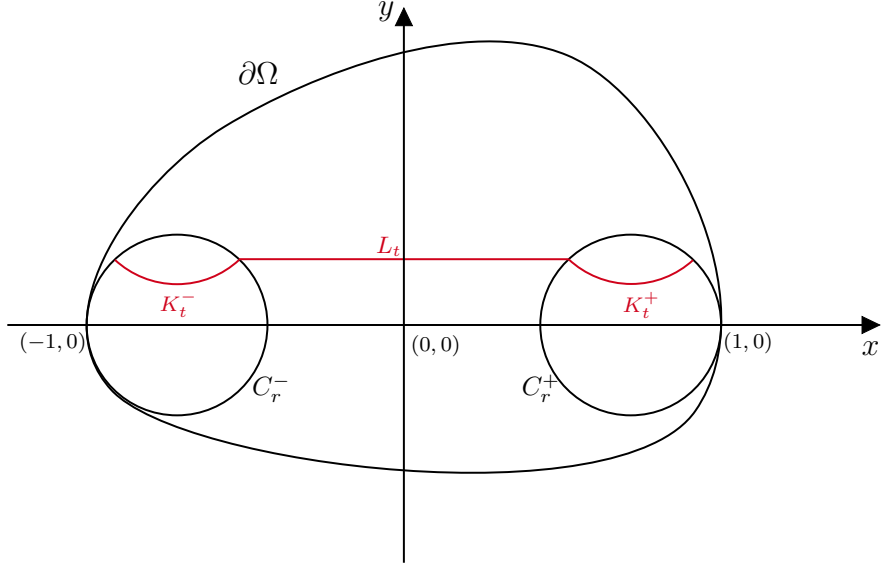
\begin{figure}[t]
\begin{center}
\tikzset{every picture/.style={line width=0.75pt}} 
\begin{tikzpicture}[x=0.75pt,y=0.75pt,yscale=-1,xscale=1]

\draw   (258.09,148.25) .. controls (302.33,121.28) and (380.53,91.33) .. (430.95,112.3) .. controls (481.37,133.27) and (533.57,236.83) .. (495.5,291.75) .. controls (457.43,346.67) and (237.01,320.19) .. (201,281.25) .. controls (164.99,242.31) and (213.84,175.21) .. (258.09,148.25) -- cycle ;
\draw    (148.5,248.4) -- (585.5,248.4) ;
\draw [shift={(588.5,248.4)}, rotate = 180] [fill={rgb, 255:red, 0; green, 0; blue, 0 }  ][line width=0.08]  [draw opacity=0] (8.93,-4.29) -- (0,0) -- (8.93,4.29) -- cycle    ;
\draw   (188.5,248.4) .. controls (188.5,223.22) and (208.92,202.8) .. (234.1,202.8) .. controls (259.28,202.8) and (279.7,223.22) .. (279.7,248.4) .. controls (279.7,273.59) and (259.28,294) .. (234.1,294) .. controls (208.92,294) and (188.5,273.59) .. (188.5,248.4) -- cycle ;
\draw   (417.3,248.4) .. controls (417.3,223.22) and (437.72,202.8) .. (462.9,202.8) .. controls (488.08,202.8) and (508.5,223.22) .. (508.5,248.4) .. controls (508.5,273.59) and (488.08,294) .. (462.9,294) .. controls (437.72,294) and (417.3,273.59) .. (417.3,248.4) -- cycle ;
\draw  [color={rgb, 255:red, 208; green, 2; blue, 27 }  ,draw opacity=1 ] (265.34,215.59) .. controls (257.11,223.18) and (246.14,227.8) .. (234.1,227.8) .. controls (222.01,227.8) and (211,223.13) .. (202.75,215.5) ;  
\draw  [color={rgb, 255:red, 208; green, 2; blue, 27 }  ,draw opacity=1 ] (494.33,215.66) .. controls (486.11,223.13) and (475.19,227.69) .. (463.2,227.69) .. controls (450.97,227.69) and (439.85,222.95) .. (431.57,215.2) ;  
\draw [color={rgb, 255:red, 208; green, 2; blue, 27 }  ,draw opacity=1 ]   (265.23,215.18) -- (431.57,215.2) ;
\draw    (348.5,368.4) -- (348.5,91.4) ;
\draw [shift={(348.5,88.4)}, rotate = 90] [fill={rgb, 255:red, 0; green, 0; blue, 0 }  ][line width=0.08]  [draw opacity=0] (8.93,-4.29) -- (0,0) -- (8.93,4.29) -- cycle    ;

\draw (578,255) node [anchor=north west][inner sep=0.75pt]    {$x$};
\draw (263.5,113.9) node [anchor=north west][inner sep=0.75pt]    {$\partial \Omega $};
\draw (334,83.4) node [anchor=north west][inner sep=0.75pt]    {$y$};
\draw (270.5,271.8) node [anchor=north west][inner sep=0.75pt]  [font=\footnotesize]  {$C_{r}^{-}$};
\draw (406.5,271.8) node [anchor=north west][inner sep=0.75pt]  [font=\footnotesize]  {$C_{r}^{+}$};
\draw (350.5,251.8) node [anchor=north west][inner sep=0.75pt]  [font=\tiny]  {$( 0,0)$};
\draw (508,250.4) node [anchor=north west][inner sep=0.75pt]  [font=\tiny]  {$(1,0)$};
\draw (153,250.4) node [anchor=north west][inner sep=0.75pt]  [font=\tiny]  {$(-1,0)$};
\draw (224,229.9) node [anchor=north west][inner sep=0.75pt]  [font=\tiny,color={rgb, 255:red, 208; green, 2; blue, 27 }  ,opacity=1 ]  {$K_{t}^{-}$};
\draw (457.5,230.4) node [anchor=north west][inner sep=0.75pt]  [font=\tiny,color={rgb, 255:red, 208; green, 2; blue, 27 }  ,opacity=1 ]  {$K_{t}^{+}$};
\draw (333,202.4) node [anchor=north west][inner sep=0.75pt]  [font=\scriptsize,color={rgb, 255:red, 208; green, 2; blue, 27 }  ,opacity=1 ]  {$L_{t}$};
\end{tikzpicture}
\end{center}
\caption{The Curves $K_{t} = K^{+}_{t} \cup L_{t} \cup K_{t}^{-}$} in $\Omega$.
\label{fig: barrier circles}
\end{figure}

\subsection{Old-but-not-ancient solutions}

For each $\rho > 0$, choose a smooth curve $\Gamma^{\rho}$ in $\Omega$ with the following properties:
\begin{enumerate}
	\item[(a)] $\Gamma^{\rho} \subset \Omega \cap \{(x,y) \in \R^{2} \mid y>0\}$.
	\item[(b)] $\Gamma^{\rho}$ meets $\partial\Omega$ orthogonally and lies below the line $y = \rho$.
	\item[(c)] $\Gamma^{\rho} \cap \Omega$ is the relative boundary of a convex region $\Omega^{\rho} \subset \Omega \cap \{y > 0\}$.
	\item[(d)] The curvature $\k^{\rho}$ of $\Gamma^{\rho}$ has a unique critical point at which $\k^{\rho}$ is minimised.  
\end{enumerate}
Such a curve always exists. In particular we have the following.

\begin{lem}\label{lem:angenant ovals existence}
	Define the scaled and horizontally shifted Angenent oval
	\[
		A^{\lambda_{\rho},\xi_{\rho}}_{t_{\rho}} := \{(x,y) \in \Omega \mid \sin(\lambda_{\rho} y) = e^{\lambda_{\rho}^{2}t_{\rho}}\cosh(\lambda_{\rho}(x-\xi_{\rho}))\}.
	\]
	Then, for all $\rho < r$, there exist $\l_\rho, \xi_\rho$ and $t_\rho$ such that the curve $\Gamma^{\rho} := A^{\lambda_{\rho},\xi_{\rho}}_{t_{\rho}}$ satisfies \emph{(a)}-\emph{(d)}. Moreover, as $\rho \to 0$, $\lambda_{\rho} \to \lambda_{0}$, where $\lambda_{0}$ is defined to be the solution of  
	\[
		\lambda^{2} - \lambda(\k^{\Omega}(e_{1}) + \k^{\Omega}(-e_{1}))\coth(2\lambda) + \k(e_{1})\k^{\Omega}(-e_{1}) = 0,
	\]
	which satisfies $\lambda > \max\{\k^{\Omega}(e_{1}), \k^{\Omega}(-e_{1})\}$. Additionally, $\xi_{\rho} \to \xi_{0}$ where
	\[
		\xi_{0} = 1 - \frac{1}{\lambda_{0}}\cosh^{-1}\left(\frac{1}{\sqrt{1-\frac{\k^{\Omega}(e_{1})^{2}}{\lambda_{0}^{2}}}} \right).
	\]
\end{lem}

The proof of the existence of such $A^{\lambda_{\rho},\xi_{\rho}}_{t_{\rho}}$ and properties thereof can be found in the appendix. \newline
The idea to construct an ancient solution is to take subsequential limits of the flows coming out of the curves $\Gamma^{\rho}$, as defined in Lemma \ref{lem:angenant ovals existence}, as $\rho \to 0$.

\begin{lem}\label{timetendsto-infty}
	For each $\Gamma^{\rho}$, with $\rho$ sufficiently small, there exists a $K_{t_{\rho}}$ as defined by \eqref{barriercurve}, such that $K_{t_{\rho}}$ lies above $\Gamma^{\rho}$ and $t_{\rho} \to -\infty$ as $\rho \to 0$.
\end{lem}
\begin{proof}
	Given $\rho<r$, we can pick $K_{t_{\rho}}$ so that it is tangent to the line $y = \rho$ at two points. It follows that $K_{t_{\rho}}$ lies above $\Gamma^{\rho}$. As $\rho \to 0$, $\omega(r^{-2}t_{\rho}) = \arcsin(e^{2(r^{-2}t_{\rho})}) \to 0$ which implies $t_{\rho} \to -\infty$.
\end{proof}

The work of Stahl \cite{convergence,regularity}, now yields the following \emph{old-but-not-ancient solutions}.

\begin{lem}\label{oldbutnotancientsolutions}
	For each $\rho > 0$, there exists a solution to the free boundary curve shortening flow $\{\Gamma_{t}^{\rho}\}_{t \in [\alpha_{\rho},0)}$ with $\Gamma^{\rho}_{\alpha_{\rho}} = \Gamma^{\rho}$. Furthermore, this solution satisfies the following:
	\begin{enumerate}
		\item $\Gamma^{\rho}_{t}$ is convex and locally uniformly convex for each $t \in (\alpha_{\rho}, 0)$.
		\item The curvature $\k^{\rho}$ of $\Gamma^{\rho}_{t}$ has only one critical point at which $\k^{\rho}$ has a minimum.
		\item $\alpha_{\rho} \to -\infty$ as $\rho \to 0$.
	\end{enumerate}
\end{lem}
\begin{proof}
	Existence of a maximal solution to curve shortening flow out of $\Gamma^{\rho}$ which meets $\partial\Omega$ orthogonally was proven by Stahl \cite{convergence, regularity}, similarly it was shown there that the solution remains convex, locally uniformly convex and shrinks to a point on the boundary at the final time. We obtain our solution $\{\Gamma_{t}^{\rho}\}_{t \in [\alpha_{\rho},0)}$ via time translation.
	By \cite{convergence}, we know that at a boundary point $|\k^{\rho}_{s}| = \k^{\rho}\k^{\Omega}$ and so an application of Sturm's theorem \cite{MR953678} to $\k^{\rho}_{s}$ proves (2).
	Finally property (3) is an immediate consequence to Lemma \ref{timetendsto-infty} and Proposition \ref{barrier}. 
\end{proof}

Next, we wish to show that we can obtain estimates for the curvature and its derivatives for the the flows $\{\Gamma_{t}^{\rho}\}_{t \in [\alpha_{\rho},0)}$ which are uniform in $\rho$.
In the following, we fix $\rho > 0$,  drop the super/sub-script $\rho$ and fix the following notation;
\begin{equation}\label{eq: definition of p}
	\underline{\k}(t) := \min_{\Gamma_{t}} \k = \k(p(t))
\end{equation}
and 
\begin{equation}\label{eq: definition of q}
	\{q^{\pm}(t)\} = \partial\Omega \cap \Gamma_{t}
\end{equation}
with $x(q^{-}) < x(q^{+})$. We define $\theta_{\pm} \in (0,\pi)$ to be such that if we parameterise $\Gamma_{t}$ via the turning angle, $\gamma_{t} = \gamma_{t}(\theta)$, as described for $\Phi$ in the beginning of subsection 2.1, then the domain of $\gamma_{t}$ is $[-\theta_{-},\theta_{+}]$. Finally, we will write 
\[
	\overline{\k}_{+}(t) := \k(q^{+}(t)), \:\: \overline{\k}_{-}(t) := \k(q^{-}(t)).
\]

\begin{lem}\label{gradientbound}
	For any old-but-not-ancient solution, with $\rho$ sufficiently small, there exists a constant $C$ such that for all $t < -\frac{2|\Omega|}{\pi}$,
	\[
		\sin\left(\frac{\theta_{+}(t) + \theta_{-}(t)}{2}\right) \leq Ce^{rt}
	\]
	where $r$ is defined as in \eqref{eq:assumptions on r}.
\end{lem}
\begin{proof}
	We use the parametrisation of $\Gamma_{t}$ by turning angle
	\[
		\gamma_{t} : [-\theta_{-}, \theta_{+}] \to \Omega,
	\]
	so that the unit tangent vector to the curve at $\gamma_{t}(\theta)$ is given by $\tau_{t}(\theta) = (\cos\theta,\sin\theta)$. Let $\underline{\theta} := \underline{\theta}(t)$ be such that $\gamma(\underline{\theta}(t)) = p(t)$.
	Since $\Gamma_{t}$ is convex, we have
	\[
	2 \geq \langle q^{+} - p, e_{1}\rangle = \int_{\underline{\theta}}^{\theta_{+}} \frac{\cos u}{\k(u)}du \geq \frac{1}{\overline{\k}_{+}}\int_{\underline{\theta}}^{\theta_{+}} (\cos u) du = \frac{\sin\theta_{+} - \sin\underline{\theta}}{\overline{\k}_{+}}
	\]
	and similarly
	\[
		2 \geq \langle p - q^{-},e_{1}\rangle = \int_{-\theta_{-}}^{\underline{\theta}}\frac{\cos u}{\k(u)}du \geq \frac{1}{\overline{\k}_{-}}\int_{-\theta_{-}}^{\underline{\theta}}(\cos u)du = \frac{\sin\underline{\theta} + \sin\theta_{-}}{\overline{\k}_{-}},
 	\]
 	which yields
 	\[
 		\overline{\k}_{+} + \overline{\k}_{-} \geq \frac{\sin\theta_{+} + \sin\theta_{-}}{2}.
 	\]
 	At the boundary $|\k_{s}| = \k\k^{\Omega}$, and so 
 	\[
 		\frac{d\theta_{+}}{dt} = \overline{\k}_{+}\k^{\Omega}, \:\: \frac{d\theta_{-}}{dt} = \overline{\k}_{-}\k^{\Omega},
 	\]
 	therefore,
 	\begin{align*}
 		\frac{d(\theta_{+} + \theta_{-})}{dt} &\geq 4r(\overline{\k}_{+} + \overline{\k}_{-}) \geq 2r(\sin\theta_{+} + \sin\theta_{-})\\
 		&= 4r\tan\left(\tfrac{\theta_{+} + \theta_{-}}{2}\right)\cos\left(\tfrac{\theta_{+} + \theta_{-}}{2}\right)\cos\left(\tfrac{\theta_{+} - \theta_{-}}{2}\right).
 	\end{align*}
 	Consider the time $t_{0}$ for which
 	\begin{equation}\label{eq: definition of t0}
 		\theta_{+}(t) + \theta_{-}(t) = \frac{\pi}{2}.
 	\end{equation}
 	Note that by considering $\rho$ sufficiently small, we can ensure that such a $t_{0} \geq \alpha_{\rho}$ exists. Thus for any $t < t_{0}$ we have
 	\[
 		\frac{d(\theta_{+} + \theta_{-})}{dt} \geq 2r\tan\left(\frac{\theta_{+} + \theta_{-}}{2}\right),
 	\]
 	which, after integrating from $t$ to $t_{0}$, yields
 	\[
 		\sin\left(\frac{\theta_{+} + \theta_{-}}{2}\right) \leq Ce^{rt}.
 	\]
 	Finally, to bound $t_{0}$, by monotonicity of $\theta_{\pm}(t)$, we have
 	\[
 		\theta_{+}(t) + \theta_{-}(t) \geq \frac{\pi}{2}
 	\]
 	for all $t \geq t_{0}$. Define $A(t)$ to be the area of the convex region enclosed by $\Gamma_{t}$ and $\partial\Omega$. Since 
 	\[
 		-\frac{dA}{dt} = \theta_{+}(t) + \theta_{-}(t),
 	\]
 	integrating from $t_{0}$ to $0$ gives us
 	\[
 		A(t_{0}) \geq -\frac{\pi}{2}t_{0},
 	\]
 	and since $A(t_{0}) \leq |\Omega|$, we obtain $-t_{0} \leq \frac{2|\Omega|}{\pi}$. 
 	\end{proof}
 
\begin{remark}\label{boundsont0} We remark here, as it will be used later, that we can also find an upper bound for $t_{0}$, the first time satisfying \eqref{eq: definition of t0}.
Indeed, consider a continuous function $h:[0,\tfrac{\pi}{2}] \to \R_{\geq 0}$ sending $\theta$ to the area of the convex region enclosed by $\partial\Omega$ and the line segment joining points $p_{\theta} = \Phi(\frac{\pi}{2} + \theta), q_{\theta} = \Phi(\pi + \theta) \in \partial\Omega$, where $\Phi(\theta)$ is the angle parametrisation defined in \eqref{eq:boundary paramatrisation}. By smoothness of the boundary, $h(\theta)$ is continuous, and $h(\theta) > 0$. By compactness of $\overline{\Omega}$, $\inf_{\theta \in [0,\frac{\pi}{2}]}h(\theta) > 0$. Convexity of $\{\Gamma_{t}\}_{t \in [\alpha_{\rho},0)}$ implies that $A(t_{0}) \geq \inf_{\theta \in [0,\frac{\pi}{2}]}h(\theta)$, independent of $\alpha_{\rho}$. Thus, by using the fact that 
 	\[
 		-\frac{dA}{dt} = \theta_{+}(t) + \theta_{-}(t) \leq \pi
 	\] 
 	for all $t<0$, integrating from $t_{0}$ to $0$ and using the fact that $A(t_{0}) \geq \inf_{\theta \in [0,\frac{\pi}{2}]}h(\theta)$, gives the upper bound for $t_{0}$.
\end{remark}

With Lemma \ref{gradientbound}, we can obtain uniform estimates which yield the following.

\begin{prop}\label{constructedancientsolution}
	For any diameter $D$ of $\Omega$, there exist two convex, locally uniformly convex, ancient solutions to the free boundary curve shortening flow in $\Omega$, one lying on each side of the diameter. As $t \to -\infty$ both of these solutions converge to $D$.
\end{prop} 
\begin{proof}
	For each $\rho$ sufficiently small, consider the old-but-not-ancient solution $\{\Gamma_{t}^{\rho}\}_{t \in [\alpha_{\rho},0)}$ as constructed in Lemma \ref{oldbutnotancientsolutions}. If we represent $\Gamma_{t}^{\rho}$ as a graph over the $x$-axis; $x \mapsto y^{\rho}(x,t)$, then convexity and Lemma~\ref{gradientbound} implies that $|y^{\rho}_{x}| = |\tan\theta|$ can be bounded uniformly in $\rho$, which in turn implies the gradient of this graph representation is uniformly bounded in $\rho$. Therefore, Stahl's (global in space, interior in time) Ecker-Huisken type estimates \cite{convergence} imply uniform-in-$\rho$ bounds for the curvature and its derivatives. Therefore, the limit
	\[
		\{\Gamma_{t}^{\rho}\}_{t \in [\alpha_{\rho},0)} \to \{\Gamma_{t}\}_{t \in (-\infty,0)}
	\]
	exists in the smooth topology (globally in space on compact subsets of time), and the limit satisfies the curve shortening flow with free boundary in $\Omega$. By our uniform bounds on $t_{0}$ in Remark~\ref{boundsont0}, it is easily verified that this limit is not the trivial solution. Since each $\Gamma_{t}$ is the limit of convex boundaries, we conclude that the limit is also convex at each time-slice, and, by \cite[Corollary 4.5]{regularity}, $\Gamma_{t}$ is also locally uniformly convex for all $t$. The second solution is achieved by repeating the construction in $\Omega \cap \{y < 0\}$.
\end{proof}

\section{Asymptotics for the Height}

In this section, we fix the ancient solution, $\{\Gamma_{t}\}_{t \in (-\infty,0)}$, that we have constructed in Lemma~\ref{constructedancientsolution} as a limit of old-but-not-ancient solutions $\{\Gamma_{t}^{\rho}\}_{t \in [-\alpha_{\rho},0)}$ with initial condition $\Gamma^{\rho}_{\alpha_{\rho}} = \Gamma^{\rho}$ as defined in Lemma \ref{lem:angenant ovals existence}. We shall aim to show that as a graph, $\lim_{t \to -\infty}e^{-\lambda_{0}^{2}t}y(x,t)$ exists in $(0,\infty)$ for all $x \in (-1,1)$. This asymptotic behaviour will be used to show uniqueness in the next section. The argument presented in this section follows the general idea of that in \cite[Section 2.4]{classification}.
However, in this case we have to account for the lack of symmetry. For instance, the point for which the minimum curvature occurs  for some time-slice of our constructed solution need not occur on the $y$-axis. Moreover, the boundary maximum principles become more complicated as the relation $\k_{s} = \k$ is no longer valid. In order to deal with the fact that $\partial\Omega$ doesn't have constant curvature, we use a Taylor expansion argument around $D$.

\begin{lem}\label{supportfunctionbound}
    For all sufficiently negative time, 
    \[
        |\langle \gamma, \nu \rangle| \leq \Lambda \k
    \]
    for some positive constant $\Lambda > 0$.
\end{lem}
\begin{proof}
    Parameterise $\gamma$ via turning angle as in Lemma \ref{gradientbound}. First we show that for sufficiently negative times, 
    \begin{equation}\label{eq:xsin bound}
    	x\sin\theta \leq C_{1}y
    \end{equation}
     for some constant $C_{1}$. At the boundary we have 
    \begin{equation}\label{eq: y of theta at the boundary}
    y(\theta) = \langle\Phi(\theta +\frac{\pi}{2}),e_{1}\rangle = \int_{\frac{\pi}{2}}^{\frac{\pi}{2}+ \theta} \frac{\sin{u}}{\k^{\Omega}(u)}du,
    \end{equation}
    where we are using the angle parametrisation $\Phi$ as defined by \eqref{eq:boundary paramatrisation}.
    Therefore, after an application of L'Hopital's rule 
    \[
        \frac{x\sin\theta}{y} \leq \frac{|\sin\theta|}{y} \overset{t \to -\infty}{\to}\k^{\Omega}(\pm e_{1}).
    \]
    Therefore, we can pick $C_{1}$ so that the \eqref{eq:xsin bound} is true at the boundary for $t$ sufficiently negative. In the interior, provided that $t << 0$, we can estimate $\k < 1$, since $\Gamma_{t}$ converges to $D$ as $t \to -\infty$, and so
    \begin{align*}
        (\partial_{t} - \Delta)(x\sin\theta - C_{1}y) &= x\k^{2}\sin\theta - 2\k\cos^{2}\theta \\
        &= \k\left(x\k\sin\theta - 2\cos^{2}\theta  \right) \\
        &\leq \k(|\sin\theta| - \cos^{2}\theta) \\
        &\leq 0,
    \end{align*}
    provided $|\theta| \leq \sin^{-1}\left({\frac{-1 + \sqrt{5}}{2}}\right)$, where we have used the estimate $x\k < 1$ in the penultimate inequality.  The maximum principle, paired with the fact that $\lim_{t \to -\infty}(x\sin\theta - C_{1}y) = 0$, proves \eqref{eq:xsin bound}.
    Next, we claim that 
    \begin{equation}\label{eq:y bound}
    	y \leq C_{2}\k
    \end{equation}
    for some constant $C_{2}$, from which the Lemma follows since 
    \[
        |\langle \gamma, \nu \rangle| = |x\sin\theta - y\cos\theta| \leq x\sin\theta + y \leq (C_{1}+1)C_{2}\k.
    \]
    To prove \eqref{eq:y bound}, note that at the boundary
    \[
    	\lim_{t \to -\infty} \frac{y}{\k} = \frac{1}{(\k^{\Omega})^{2}},
    \]
    which follows once again after an application of L'Hopital's rule, where we use the fact that at the boundary $\k_{\theta} = \k^{\Omega}$, and \eqref{eq: y of theta at the boundary}.
    Therefore, we can find a constant $C_{2}$ so that \eqref{eq:y bound} is true at the boundary. In the interior, 
    \[
    	(\partial_{t} - \Delta)(y-C_{2}\k) = -C_{2}\kappa^{3} < 0,
    \]
    and so the maximum principle along with the fact that $\lim_{t \to -\infty}(y-C_{2}\k) = 0$ proves \eqref{eq:y bound}.
\end{proof}

Next we will prove curvature estimates, which will allow us to obtain estimates of the height function. As introduced in Section 2, we will use the notation $\overline\k_{\pm}$, $\theta_{\pm}$, but now for the ancient solution $\{\Gamma_{t}\}_{t \in (-\infty,0)}$ instead of the old-but-not-ancient solutions $\{\Gamma_{t}^{\rho}\}_{t \in [\alpha_{\rho,0})}$. Similarly, we define $\overline{\k} = \max\{\overline\k_{+},\overline{\k_{-}}\}$.

\begin{lem}\label{curvatureboundlemma}
	There exists a constant $C$, so that for sufficiently negative time, 
	\[
		\underline{\k} \leq Ce^{rt},
	\]
	where $r$ is defined in \eqref{eq:assumptions on r}.
\end{lem}
\begin{proof}
	Let $t < 0$, and let $q_{+}, q_{-}$ and $p$ be defined by \eqref{eq: definition of p} and \eqref{eq: definition of q}. Then,
	\[
		\langle q_{+} - p, e_{1} \rangle = \int_{\underline{\theta}}^{\theta_{+}}\frac{\cos(u)}{\k}du \leq \frac{\sin \theta_{+} - \sin \underline{\theta}}{\underline{\k}}
	\]
	and
	\[
		\langle p - q_{-}, e_{1} \rangle = \int^{\underline{\theta}}_{-\theta_{-}}\frac{\cos(u)}{\k}du \leq \frac{\sin \underline{\theta} + \sin \theta_{-}}{\underline{\k}}.
	\]
	After adding these inequalities, we obtain
	\[
		\underline{\k} \leq \frac{\sin\theta_{+} + \sin\theta_{-}}{\langle q_+ - q_-, e_{1} \rangle} \leq C\sin\left(\frac{\theta_{+} + \theta_{-}}{2}\right)
	\]
	for sufficiently negative time. The result then follows from Lemma~\ref{gradientbound}.
\end{proof}

Lemma \ref{curvatureboundlemma} allows us to obtain the following sharper estimates.

\begin{lem}\label{curvatureestimatesfortheheight}
    There exist constants $C_{1},C_{2}$ so that for sufficiently negative time 
    \[
        \overline{\k} \leq C_{1}e^{rt} 
    \]
    and
    \[
        |\k_{s}| \leq C_{2}\k,
    \]
    where $r$ is as defined in \eqref{eq:assumptions on r}.
\end{lem}
\begin{proof}
    First note that the first inequality follows from the second; if $|\k_{s}| \leq C_{2}\k$, then $(\log\k)_{s} \leq C_{2}$. By integrating from the point of minimum curvature to the point of maximum curvature and noting that $\text{Length}(\Gamma_{t}) \leq 2$, we have for small enough $t$
    \[
        2C_{2} \geq \log\frac{\overline{\k}}{\underline{\k}} ,
    \]
    which implies $\overline{\k} \leq e^{2C_{2}}\underline{\k}$. The first inequality then follows from Lemma~\ref{curvatureboundlemma}. For the second inequality, by Lemma~\ref{supportfunctionbound}, it suffices to show that $|\k_{s}| - C\k + \langle \gamma, \nu \rangle \leq 0$ for some constant $C$. For each $0 < \varepsilon \leq 1$, define a function
    \[
        f_{\varepsilon} = |\k_{s}| - C\k + \langle \gamma, \nu \rangle - \varepsilon(e^{t} + 1).
    \]
    Clearly $\lim_{t \to -\infty} f_{\varepsilon} = -\varepsilon < 0$. We will show that for sufficiently negative times, independent of $\varepsilon$, $f_{\varepsilon}$ remains negative. First of all, we can ensure that $f_{\varepsilon} < 0$ at the boundary, since, by using Lemma~\ref{supportfunctionbound}, 
    \begin{align*}
        f_{\varepsilon} &\leq \k\k^{\Omega} - C\k + \langle \gamma, \nu \rangle \\
        &\leq \k(\k^{\Omega} - C + \Lambda),
    \end{align*}
    and so we can pick $C$ to ensure that $f_{\varepsilon}$ is always negative on the boundary, and we also ensure that $3C + \Lambda > 1$. Now assume there exists a first time at which $f_{\varepsilon} = 0$ at some point. Such a point must be in the interior and moreover, since $C > \Lambda$, at that point we must have $\k_{s} \neq 0$ and without loss of generality we assume $\k_{s} > 0$. If we denote by $T_{C}$ the time for which $\k < \frac{1}{2(3C+\Lambda)}$ for all $t < T_{C}$, then for all $t < T_{C}$,  at the first point for which $f_{\varepsilon} = 0$, we have
    \begin{align*}
        0 &\leq (\partial_{t} - \Delta)f_{\varepsilon}  \\ 
        &\leq 4\k^{2}\k_{s} - C\k^{3} + \k^{2}\langle \gamma, \nu \rangle - 2\k - \varepsilon e^{t} \\
        &\leq 4\k^{2}(C\k - \langle \gamma, \nu \rangle + \varepsilon(e^{t} + 1)) - C\k^{3} + \k^{2}\langle \gamma, \nu \rangle - 2\k - \varepsilon e^{t} \\
        &= 3C\k^{3} + 4\k^{2}\varepsilon + (4\k^{2} - 1)\varepsilon e^{t} - 3\k^{2}\langle \gamma, \nu \rangle - 2\k \\
        &\leq (3C + \Lambda)\k^{3} + 4\k^{2} + (4\k^{2} - 1)\varepsilon e^{t} - 2\k \\ 
        &< 0
    \end{align*}
    for sufficiently negative time, where we have used used the Lemma~\ref{supportfunctionbound} in the penultimate inequality above. This is a contradiction which completes the proof.
\end{proof}

\begin{remark}\label{notdependentonconstructedsolution}
	Notice that proof of Lemmas \ref{supportfunctionbound}, \ref{curvatureboundlemma} and \ref{curvatureestimatesfortheheight} did not depend on the solution constructed in Lemma \ref{constructedancientsolution} and are therefore true for \emph{any} ancient solution to the free boundary curve shortening flow in $\Omega$.
\end{remark}

With these curvature estimates in mind, we are ready to prove the following height estimates. These height estimates do depend on the particular ancient solution constructed in Lemma \ref{constructedancientsolution}.

\begin{lem}\label{height1}
    There exists a positive constant $n$ depending only on the boundary curve $\Omega$ such that the ancient solution $\{\Gamma_{t}\}_{t \in (-\infty,0)}$ satisfies
    \[
        \frac{\k}{y}e^{ny} \geq \lambda_{0}^{2} - ne^{rt}
    \]
    for sufficiently negative time.
\end{lem}
\begin{proof}
    We will prove that $\frac{\k}{y}e^{ny} \geq \lambda^{2} - ne^{ct}$ on each old-but-not-ancient solution $\{\Gamma_{t}^{\lambda}\}_{t \in (\alpha_{t}^{\lambda},0)}$ for $\lambda$ sufficiently close to $\lambda_{0}$. Indeed, on the initial curve we have 
    \[
    	\frac{\k}{y}e^{ny} \geq \lambda^{2}\cos\theta \geq \lambda^{2}(1-\sin^{2}\theta) \geq \lambda^{2}-Ce^{2rt},
    \]
    where we have used Lemma~\ref{gradientbound} in the last inequality. Thus, the bound is always true at the initial curve, for sufficiently large $-\alpha_{t}^{\lambda}$, provided $n \geq C$. 
    We will now show that the bound remains true during the flow for at least a small period of time independent of $\lambda$. Presume there is a first time in which $\frac{\k}{y}e^{ny} + ne^{rt} = \lambda^{2}$. If such a point were to occur in the interior, then
    \[
    	\left(\frac{\k}{y}e^{ny}\right)_{s} = 0,
    \]
    which means 
    \[
    	\frac{(\k e^{ny})y_{s}^{2}}{y^{3}} = \frac{(\k e^{ny})_{s}y_{s}}{y^{2}}
    \]
    at such a point. Therefore
    \begin{align*}
    	0 &\geq (\partial_{t} - \Delta)\left(\frac{\k}{y}e^{ny} + ne^{rt}\right) \\
    	&= \frac{1}{y}(\partial_{t} - \Delta)\left(\k e^{ny}\right) + nc e^{rt} \\
		&\geq \k ne^{ny}\left(-n\sin^{2}\theta - 2\frac{\k_{s}}{\k}\sin\theta\right) + nc e^{rt} \\
		&\geq \k ne^{ny}(-n\sin^{2}\theta - 2C_{2}\sin\theta) + nc e^{rt}
    \end{align*}
    for sufficiently negative time, where we have used Lemma~\ref{curvatureestimatesfortheheight} in the last inequality. Using Lemma~\ref{gradientbound}, we can estimate $\sin\theta \leq Ce^{rt}$ and so 
    \begin{align*}
		0 &\geq (\partial_{t} - \Delta)\left(\frac{\k}{y}e^{ny} + ne^{rt}\right) \\
		&\geq e^{rt}\left(nr - 2C_{2} - nC^{2}e^{rt}\right),
	\end{align*}
	which can be made positive for sufficiently negative time, provided we adjust $n$ so that $nr > 2C_{2}$. This is a contradiction and so the first time in which $\frac{\k}{y}e^{ny} + ne^{rt} = \lambda^{2}$ cannot happen in the interior. 
	Similarly, if we presume that this minimum occurs at the right-most boundary point, then the Hopf boundary point lemma implies
	\begin{equation}\label{eq:tayloresponsionterm}
		0 > \left(\frac{\k}{y}e^{ny} + ne^{rt}\right)_{s} = \frac{\k}{y}e^{ny}\left(\k^{\Omega} - \frac{\sin\theta}{y} + n\sin\theta\right).
	\end{equation}
    However, we claim the right-hand side can be made positive by choosing $n$ sufficiently large.
    Indeed, by considering angle parametrisation $\Phi(\theta)$ as defined by \eqref{eq:boundary paramatrisation}, and the Taylor expansion of the height function on the boundary curve, $\langle \Phi(\frac{\pi}{2} + \theta),e_{2}\rangle$, around $\theta = 0$, with respect to $\theta$, we obtain
    \[
    	y(\theta) = \frac{\theta}{\k^{\Omega}_{+}} + O(\theta^{2})
    \]
    from which, it follows
    \[
    	\k^{\Omega} - \frac{\sin\theta}{y} \geq -C\sin\theta
    \]
    for some constant $C$. Therefore, 
    \begin{align*}
    	0 &> \left(\frac{\k}{y}e^{ny} + ne^{ct}\right)_{s} \\
    	&= \frac{\k}{y}e^{ny}\left(\k^{\Omega} - \frac{\sin\theta}{y} + n\sin\theta\right) \\
    	&\geq 0,
    \end{align*}
    so by choosing $n > C$, and the claim follows. The same argument after reflecting about the $y$-axis demonstrates that this minimum cannot occur at the left-most boundary point either. 
	This completes the proof.
\end{proof}

\begin{lem}\label{height2}
	There exists a positive constant $m$ depending only on the boundary curve $\Omega$ such that the ancient solution $\{\Gamma_{t}\}_{t \in (-\infty,0)}$ satisfies
    \[
        \frac{\k}{y}e^{-ny} \leq \lambda_{0}^{2} + me^{2rt}
    \]
    for sufficiently negative time.
\end{lem}
\begin{proof}
	We will prove the result on each old-but-not-ancient solution $\{\Gamma_{t}^{\lambda}\}_{t \in (\alpha_{t}^{\lambda},0)}$ where $\lambda$ is sufficiently close to $\lambda_{0}$. For sufficiently negative time we have
	\begin{align*}
		(\partial_{t} - \Delta)\left(\frac{\k}{y}e^{-ny}\right) &= \frac{1}{y}\left(\k^{3}e^{-ny} - \k n^{2}\sin^{2}\theta e^{-ny}\right) + 2\left\langle \nabla \left(\frac{\k}{y}e^{-ny}\right), \frac{\nabla y}{y}\right\rangle \\
		&\leq \k^{2}\left(\frac{\k}{y}e^{-ny}\right) + 2\left\langle \nabla \left(\frac{\k}{y}e^{-ny}\right), \frac{\nabla y}{y}\right\rangle \\
		&\leq C_{1}^{2}e^{2rt}\left(\frac{\k}{y}e^{-ny}\right) + 2\left\langle \nabla \left(\frac{\k}{y}e^{-ny}\right), \frac{\nabla y}{y}\right\rangle,
	\end{align*}
	where we have used Lemma~\ref{curvatureestimatesfortheheight} in the last inequality. At the boundary, we have 
	\[
		\left(\frac{\k}{y}e^{-ny}\right) \leq \frac{\k}{y}e^{-ny}\left(\k^{\Omega} - \frac{\sin\theta}{y} - n\sin\theta\right)
	\]
	which is negative by our choice of $n$ in the Lemma~\ref{height1}. Hence the Hopf boundary point lemma and the ODE comparison principle imply
	\[
		\max_{\Gamma_{t}}\frac{\k}{y}e^{-ny} \leq C\max_{\Gamma_{\alpha_{t}}}\frac{\k}{y}e^{-ny}.
	\]
	But now,
	\[
		(\partial_{t} - \Delta)\left(\frac{\k}{y}e^{-ny}\right) \leq Ce^{2rt}\max_{\Gamma_{\alpha_{t}}}\frac{\k}{y}e^{-ny} + 2\left\langle \nabla \left(\frac{\k}{y}e^{-ny}\right), \frac{\nabla y}{y}\right\rangle,
	\]
	which by the ODE comparison principle again, implies 
	\[
		\max_{\Gamma_{t}}\frac{\k}{y}e^{-ny} \leq (1 + Ce^{2rt})\max_{\Gamma_{\alpha_{t}}}\frac{\k}{y}e^{-ny}.
	\]
	Since on the initial time-slice $\Gamma_{a_{t}} = A_{t}^{\lambda, \xi}$, 
	\[
		\frac{\k}{y}e^{-ny} = \frac{\lambda\tan(\lambda y)}{y}\cos\theta e^{-ny},
	\]
	the claim follows by letting $\lambda \to \lambda_{0}$.
\end{proof}

We are now able to prove the major result of this section.

\begin{prop}\label{importantprop}
	If we parameterise $\Gamma_{t}$ as a graph over the $x$-axis, then the limit 
	\[
		A(x) := \lim_{t \to -\infty}e^{-\lambda_{0}^{2}t}y(x,t)
	\]
	exists in $(0,\infty)$ for all $x \in (-1,1)$ on the constructed ancient solution.
\end{prop}
\begin{proof}
	First we show that 
	\begin{equation}\label{aissmall}
		e^{-\lambda_{0}^{2}t}y(x,t) < C
	\end{equation}
	for some constant $C$ and for all $t$ small enough. By Lemma~\ref{height1}, for sufficiently negative time,
	\begin{align*}
		(\log y(t) - \lambda_{0}^{2}t)_{t} &= \frac{\k}{y\cos\theta} - \lambda_{0}^{2} \\ 
		&\geq (e^{-ny} - 1)\lambda_{0}^{2} - ne^{rt}.
	\end{align*}
	We claim that there exists a positive constant $a$ such that 
	\[
		f(y,t) := 1 - ae^{rt} - e^{-ny} \leq 0
	\] 
	for all $t$ sufficiently negative. Indeed $\lim_{t \to -\infty}f = 0$, and for sufficiently negative time, 
	\begin{align*}
		f_{t} &= -ar e^{rt} + \frac{n\k}{\cos\theta}e^{-ny} \\
		&\leq -ar e^{rt} + 2nCe^{rt},
	\end{align*}
	where we have used Lemma \ref{curvatureestimatesfortheheight} and have estimated $\sec\theta < 2$. Thus, we can pick $a$ so that there exists a $T < 0$ so that for any $t < T$, we have $f_{t} \leq 0$. This proves the claim, and so
	\begin{align*}
		(\log y(t) - \lambda_{0}^{2}t)_{t} &\geq -(a\lambda_{0}^{2}+n)e^{rt}.
	\end{align*}
	Integrating from $t$ to $T$ proves that $\log y(t) - \lambda_{0}^{2}t$ is uniformly bounded from above for all $t < T$, hence (\ref{aissmall}) is true for all $t$ sufficiently small. 
	Now we will prove 
	\begin{equation}\label{aisbig}
		e^{-\lambda_{0}^{2}t}y(x,t) > \tilde{C}
	\end{equation}
	for some constant $\tilde{C} > 0$ and for all $t$ small enough.
	By Lemma~\ref{height2},  for sufficiently negative time, 
	\begin{align*}
		(\log y(t) - \lambda_{0}^{2}t)_{t} &= \frac{\k}{y\cos\theta} - \lambda_{0}^{2} \\ 
		&\leq (\frac{e^{ny}}{\cos\theta} - 1)\lambda_{0}^{2} + 4me^{2rt}.
	\end{align*}
	Similar to before, we claim that there is a positive constant $b$ such that 
	\[
		g := 1 + be^{2rt} - \frac{e^{ny}}{\cos\theta} \geq 0
	\]
	for all $t$ sufficiently negative.
	Indeed, $\lim_{t \to -\infty}g = 0$ and, for sufficiently negative time, 
	\begin{align*}
		g_{t} &= - \frac{e^{ny}n\k}{\cos^{2}\theta} - \frac{e^{ny}\theta_{t}\sin\theta}{\cos^{2}\theta} + 2br e^{2rt} \\ 
		&\geq -8nC_{1}e^{rt} - 16Ce^{rt} + 2br e^{2rt},
	\end{align*} 
	where we have used the fact that $\theta_{t} = \k_{s}$, Lemma~\ref{gradientbound}, Lemma~\ref{curvatureestimatesfortheheight} and have used the estimates $-\frac{1}{\cos\theta} > -2$, $-e^{ny} > -2$ and $-\theta_{t} \geq -2$. Thus, we can pick $b > 0$ so that there exists $T < 0$ so that for any $t < T$, $g_{t} \geq 0$. This proves the claim, and so 
	\[
		(\log y(t) - \lambda_{0}^{2}t)_{t} \leq (b\lambda_{0}^{2}+4m)e^{rt}.
	\]
	Integrating from $t$ to $T$ proves that $\log y(t) - \lambda_{0}^{2}t$ is bounded from below for all $t < T$, hence (\ref{aisbig}) is true for all $t$ sufficiently small.
	Now, we show that the limit exists, from which, our uniform bounds above imply the result. By a direct calculation
	\begin{align*}
		\frac{d}{dt}\left(e^{-\lambda_{0}^{2}t}y(x,t)\right) &= \left(\frac{\k}{y\cos\theta} - \lambda_{0}^{2}\right)ye^{-\lambda_{0}^{2}t} \\
		&\geq \tilde{C}\left((e^{-ny} - 1)\lambda_{0}^{2} - ne^{rt} \right)\\
		&\geq \tilde{C}\left(-a\lambda_{0}^{2} - n \right)e^{rt},
	\end{align*}
	where we have used \eqref{aisbig} and Lemma \ref{height1}. Therefore, we conclude that
	\[
		\lim_{t \to -\infty} e^{-\lambda_{0}^{2}t}y(x,t)
	\]
	exists in $(0,\infty)$. 
\end{proof}

\section{Uniqueness}

Let $\{\Gamma_{t}\}_{t \in (-\infty,0)}$ by \emph{any} convex, locally uniformly convex ancient solution to the free boundary curve shortening flow in $\Omega$. We may assume by Stahl's theorem \cite{regularity} that $\Gamma_{t}$ contracts to a point on the boundary $\partial\Omega$ as $t \to 0$.

\begin{lem}
	$\Gamma_{t}$ converges in $C^{\infty}$ to a diameter as $t \to -\infty$.
\end{lem}
\begin{proof}
	The proof of this is the same as that shown in \cite[Lemma 3.1]{classification},  \emph{mutatis mutandis}.
\end{proof}

Therefore, by scaling and translating, we may assume without loss of generality that the backwards limit is $D=[-1,1]$ as assumed thus far. 

\begin{lem}\label{samelemmaagain}
	There exists a constant $C$ such that for sufficiently negative time
	\[
		\overline{\k} \leq C \underline{\k}
	\]
	on the ancient solution $\Gamma_{t}$.
\end{lem}
\begin{proof}
	The proof of this is identical to that in Lemma~\ref{curvatureestimatesfortheheight} as per Remark~\ref{notdependentonconstructedsolution}.
\end{proof}

\begin{lem}\label{importantpronongeneralsolution}
	If we parameterise $\Gamma_{t}$ as a graph over the $x$-axis, then there is a constant $C$ such that
	\[
		\sup_{x \in (-1,1)}\left(\limsup_{t \to -\infty} e^{-\lambda_{0}^{2}t}y(x,t)\right) < C,
	\]
\end{lem}
\begin{proof}
	Denote by $\{\hat{\Gamma}_{t}\}_{t \in (-\infty,0)}$ the solution constructed in Lemma~\ref{constructedancientsolution}. Then for all sufficiently negative time, $\Gamma_{t} \cap \hat{\Gamma}_{t} \neq \emptyset$. Indeed, if this were not true, then they would be disjoint for all times and therefore have to contract to the same point on the boundary at $t = 0$, which contradicts the avoidance principle. Therefore, for any $t$ sufficiently negative, there exists an $x_{0} \in (-1,1)$ such that $y(x_{0},t) < \hat{y}(x_{0},t)$, and therefore, by Proposition~\ref{importantprop}
	\[
		e^{-\lambda_{0}^{2}t}y(x_{0},t) \leq e^{-\lambda_{0}^{2}t}\hat{y}(x_{0},t) < C,
	\]
	Now let $x$ be some other point in $(-1,1)$. Then, by Lemma~\ref{samelemmaagain}, there exists a constant $\tilde{C}$ such that for sufficiently negative time
	\[
		y(x,t) = \int_{-\infty}^{t}y_{t}(x,t)dt = \int_{-\infty}^{t}\frac{\k(x,t)}{\cos\theta}dt \leq \tilde{C}\int_{-\infty}^{t}\frac{\k(x_{0},t)}{\cos\theta_{0}}dt = \tilde{C}y(x_{0},t).
	\]
	Hence, by once again using Proposition~\ref{importantprop}
	\[
		e^{-\lambda_{0}^{2}t}y(x,t) \leq \tilde{C}e^{-\lambda_{0}^{2}t}y(x_{0},t) < C.
	\]
\end{proof}

We will now examine the limiting behaviour of the height function on the general ancient solution.

\begin{prop}
	For some constant $A$, we have
	\[
		e^{\lambda_{0}^{2}t}y(x,t) \to A \left(\cosh(\lambda_{0}x) + \frac{\k_{1} - \k_{2}}{2\lambda_{0} - (\k_{1} + \k_{2})\tanh{\lambda_{0}}}\sinh(\lambda_{0}x) \right)
	\]
	uniformly as $t \to -\infty$, where $\k_{1} := \k^{\Omega}(1,0)$, $\k_{2} := \k^{\Omega}(-1,0)$.
\end{prop}
\begin{proof}
	For each $\tau < 0$, let $y^{\tau}(x,t) := e^{-\lambda_{0}^{2}\t}y(x,t+\tau)$ defined on the time-translated flow $\{\Gamma_{t}^{\tau}\}_{t \in (-\infty,-\tau)}$ where $\Gamma_{t}^{\tau} := \Gamma_{t+\tau}$. Lemma~\ref{importantpronongeneralsolution}, implies a uniform bound for $y^{\tau}$ on $\{\Gamma_{t}^{\tau}\}_{t \in (-\infty,T]}$ for any $T \in \R$. Alaoglu's theorem therefore yields a sequence of times $\tau_{j} \to -\infty$ such that $y^{\tau_{j}}$ converges in the weak$^*$ topology as $j \to \infty$ to some $y^{\infty} \in L^{2}_{\text{loc}}([-1,1] \times (-\infty,\infty))$. Since convexity and the boundary condition imply a uniform bound for $\nabla^{\tau}y^{\tau}$ on any time interval of the form $(-\infty,T]$, where $\nabla^{\tau}$ is the gradient on $\Gamma_{t}^{\tau}$, we may also arrange that the convergence in uniform in space at time zero, say. 
    For any $j$, note that $y^{\tau_{j}}$ satisfies the following boundary value problem;
    \[
    \begin{cases}\tag{14}\label{bvp}
        (\partial_{t} - \Delta^{\tau_{j}})y^{\tau_{j}} = 0 \:\:\: \text{in} \:\: \Gamma_{t}^{\tau}\\
        \langle \nabla^{\tau_{j}}y^{\tau_{j}},N \rangle = y \cdot f \:\:\: \text{on} \:\: \partial\Gamma_{t}^{\tau},
    \end{cases}
    \]
    where $f = \frac{\sin\theta}{y}$, $N$ is the outward unit normal to $\partial\Omega$ and $\Delta^{\tau}$ is the Laplacian on $\Gamma_{t}^{\tau}$. Since $y^{\tau_{j}}$ satisfies (\ref{bvp}) then necessarily
    \[
        \int_{-\infty}^{\tau_{j}}\int_{\Gamma_{t}^{\tau_{j}}}y^{\tau_{j}}(\partial_{t} - \Delta^{\tau_{j}})^{*}\eta = 0
    \]
    for all smooth $\eta$ which are compactly supported in time and satisfy 
    \[
        \nabla^{\tau}\eta \cdot N = \eta \cdot f \:\:\: \text{on} \:\: \partial\Gamma_{t}^{\tau_{j}},
    \]
   where $(\partial_{t} - \Delta^{\tau_{j}})^{*} = -(\partial_{t} + \Delta^{\tau_{j}})$ is the formal $L^{2}$-adjoint of the heat operator. Since $\{\Gamma_{t}^{\tau_{j}}\}_{t \in (-\infty,-\tau_{j})}$ converges uniformly in the smooth topology to the stationary interval $\{[-1,1] \times \{0\}\}_{t \in (-\infty,\infty)}$ as $j \to \infty$, we may parameterise each flow $\{\overline{\Gamma}_{t}^{\tau_{j}}\}_{t \in (-\infty,-\tau_{j})}$ over $I := [-1,1]$ by a family of embeddings $\gamma_{t}^{j}:I \times (-\infty,-\tau_{j}) \to \Omega$ which converge in $C^{\infty}_{\text{loc}}(I \times (-\infty,\infty))$ to the stationary embedding $\Gamma^{\infty}$, which is characterised by $(x,t) \mapsto xe_{1}$. Given $\eta \in C_{0}^{\infty}(I \times (-\infty, \infty))$ satisfying $\eta_{z}(1) = \eta\k_{1}$, and $\eta_{z}(-1) = -\eta\k_{2}$ (recall that $f(e_{1}) = \k_{1}$ and $f(-e_{1}) = -\k_{2}$). Set $\eta^{j} = \phi^{j}\eta$, where $\phi^{j}:[-1,1] \times (-\infty, \tau_{j}) \to \R$ is defined by
   \[
        \phi^{j}(z,t) = e^{s^{j}(z,t)},
   \]
   where $s^{j}_{z}(z,t) = (|\gamma^{j}_{z}(z,t)| - 1)f(\gamma^{j}(z))$. This ensures that $\nabla^{\tau_{j}}\eta^{j} \cdot N = \eta^{j} \cdot f$, and hence
    \[
        \int_{-\infty}^{\tau_{j}}\int_{\Gamma_{t}^{\tau_{j}}}y^{\tau_{j}}(\partial_{t} - \Delta^{\tau_{j}})^{*}\eta^{j} = 0.
    \]
    Since $\phi^{j} \to 1$ in $C^{\infty}_{\text{loc}}(I \times (-\infty,\infty))$, then weak$*$ convergence of $y^{\tau_{j}}$ to $y^{\infty}$ as $j \to \infty$ implies
    \[
        \int_{-\infty}^{\infty}\int_{\Gamma^{\infty}}y^{\infty}(\partial_{t} - \Delta^{\tau_{j}})^{*}\eta = 0.
    \]
    Thus by the $L^{2}$ theory for the heat equation, $y^{\infty}$ satisfies
    \[
    \begin{cases}
        y^{\infty}_{t} = y^{\infty}_{xx} \:\:\: \text{in} \:\: [-1,1]\\
        y^{\infty}_{x}(\pm 1) = \pm y \cdot f((\pm1,0)) \:\:\: \text{on} \:\: \partial\Gamma_{t}^{\tau}.
    \end{cases}
    \]
    Finally, we characterise the limit. Separation of variables leads us to consider the problem 
    \[
    \begin{cases}
        -\varphi_{xx} = \mu\varphi \:\:\: \text{in} \:\: [-1,1]\\
        \varphi_{x}(\pm 1) = \pm \varphi \cdot f((\pm1,0)) \:\:\: \text{on} \:\: \partial\Gamma_{t}^{\tau}.
    \end{cases}
    \]
    After a long calculation, one finds that the negative eigenspace restricted to convex functions is one dimensional. This eigenspace corresponds to the eigenvalue $\lambda_{0}$, as defined in Lemma \ref{lem:angenant ovals existence} and the corresponding eigenfunction is given by
    \[
        \varphi_{\lambda_{0}} := \cosh{\lambda_{0}x} + \frac{\k_{1} - \k_{2}}{2\lambda_{0} - (\k_{1} + \k_{2})\tanh\lambda_{0}}\sinh{\lambda_{0}x}.
    \]
    Note that in general, there might be a second negative eigenvalue, however, in that case the corresponding eigenfunction is not convex.
    Thus, 
    \[
        y^{\infty}(x,t) = Ae^{\lambda_{0}^{2}t}\left(\cosh{\lambda_{0}x} + \frac{\k_{1} - \k_{2}}{2\lambda_{0} - (\k_{1} + \k_{2})\tanh\lambda_{0}}\sinh{\lambda_{0}x}\right)
    \]
    for some $A \geq 0$, and by the avoidance principle, such an $A$ is unique.
\end{proof}

Uniqueness of the constructed ancient solution now follows directly from the avoidance principle.

\begin{thm}
	Modulo time translation, for each diameter of $\Omega$, there exists precisely two convex, locally uniformly convex, ancient solution to the free boundary curve shortening flow in $\Omega$, one lying on each side of the diameter.
\end{thm}
\begin{proof}
	Consider two convex ancient solutions $\{\Gamma_{t}\}$ and $\{\Gamma_{t}'\}$ to \eqref{eq:fbcsf} lying on one side of $D$. Given $\t > 0$, consider the time-translated solution $\{\Gamma_{t}^{\t}\}$ defined by $\Gamma_{t}^{\t} = \Gamma_{t+\t}'$. By the previous proposition;
	\[
		e^{-\lambda_{0}^{2}t}y^{\t}(x,t) \to Ae^{\lambda_{0}^{2}\t}\left(\cosh(\lambda_{0}x) + \frac{\k_{1} - \k_{2}}{2\lambda_{0} - (\k_{1} + \k_{2})\tanh\lambda_{0}}\sinh(\lambda_{0}x) \right)
	\]
	as $t \to -\infty$. Thus, $\Gamma_{t}^{\t}$ lies above $\Gamma_{t}$ for $t$ sufficiently negative. The avoidance principle then ensures that $\Gamma_{t}^{\t}$ lies above $\Gamma_{t}$ for all $t \in (-\infty,-\t)$. Taking $\t \to 0$, we find that $\Gamma'_{t}$ lies above $\Gamma_{t}$ for all $t < 0$ by the avoidance principle. But both curves reach the same point at time zero, and so they must intersect for all $t < 0$. The strong maximum principle then implies the two solutions coincide for all $t$. 
\end{proof}


\newpage

\begin{appendix}
\addtocontents{toc}{\setcounter{tocdepth}{1}}

\section{Orthogonally Intersecting Angenent Ovals}

Let $\Omega \subset \R^{2}$ be a compact, strictly convex domain and let $D$ be any diameter of $\Omega$. By shifting, rotating and dilating, we may assume that $\overline{D} = [-1,1]$. We will show that for any $\r > 0$, there exists a time-slice of an $x$-shifted Angenant oval of the form 
\[
    A_{t}^{\lambda, \xi} := \{(x,y) \in \R \times (0,\tfrac{\pi}{2\lambda}) \mid \sin(\lambda y) = e^{\lambda^{2}t}\cosh(\lambda(x - \xi))\}
\]
such that $A_{t}^{\lambda, \xi}$ intersects $\partial\Omega$ orthogonally at two points that lie below the line $y = \r$. Moreover we will show that the scale $\lambda$ has a limit as $\rho\to 0$. We will do the construction in two parts, first by asserting that we have orthogonality at a single point on the boundary, and then demonstrating we have a large enough degree of freedom to achieve orthogonality at a second point on the boundary.

\subsection{Orthogonality at a Single Point}

We adopt a graph parametrisation for $\partial\Omega \cap \{(x,y) \in \R^{2} \mid y \geq 0\}$, $\phi(x)$ where $x \in [-1,1]$. Pick a point on $\partial\Omega \cap \{(x,y) \in \R^{2} \mid y \geq 0\}$, say $(x_{0},\phi(x_{0}))$, such that $x_{0} > 0$ and $\phi'(x_{0}) < 0$. Then it is easily verified that $A_{t}^{\lambda, \xi}$ passes through $(x_{0},\phi(x_{0}))$ for any \emph{`valid'} $\lambda$ (the precise meaning of this will be established in the next lemma) and for any $t < 0$ where $\xi$ is given by
\[
    \xi = x_{0} - \frac{1}{\lambda}\cosh^{-1}\left( e^{-\lambda^{2}t}\sin(\lambda \phi(x_{0})) \right) \tag{1}\label{originaleqationforxi}.
\]
Thus, we obtain a 2-parameter family of Angenant ovals which pass through the point $(x_{0},\phi(x_{0}))$. We reduce this to a 1-parameter family, by enforcing orthogonality at that point.

\begin{lem}\label{Mainappendixlemma}
    For each $\lambda \in \left( \frac{1}{\phi(x_{0})}\tan^{-1}(\frac{-1}{\phi'(x_{0})}), \frac{\pi}{2\phi(x_{0})} \right)$, let $A_{t}^{\lambda, \xi}$ be the scaled Angenant oval where 
    \[
        t = \frac{1}{2\lambda^{2}}\log\left( \sin^{2}(\lambda \phi(x_{0})) - \frac{\cos^{2}(\lambda \phi(x_{0}))}{\phi'(x_{0})^{2}} \right) \tag{2}\label{equationfort}
    \]
    and
    \[
        \xi = x_{0} - \frac{1}{\lambda}\cosh^{-1}\left( \frac{\sin(\lambda \phi(x_{0}))}{\sqrt{\sin^{2}(\lambda \phi(x_{0})) - \frac{\cos^{2}(\lambda \phi(x_{0}))}{\phi'(x_{0})^{2}}}} \right) \tag{3}\label{equationforxi}.
    \]
    Then $A^{\lambda, \xi}_{t}$ intersects $\partial\Omega$ orthogonally at $(x_{0}, \phi(x_{0}))$.
\end{lem}
\begin{proof}
    It is easy to verify that $(\tanh(\lambda(x_{0} - \xi)), -\cot(\lambda \phi(x_{0})))$ is normal to the Angenant oval at $(x_{0},\phi(x_{0}))$. Similarly $(-\phi'(x_{0}),1)$ is normal to $\partial\Omega$ at $(x_{0},\phi(x_{0}))$. Hence, we require 
    \[
        \tanh(\lambda(x_{0} - \xi))\phi'(x_{0}) + \cot(\lambda \phi(x_{0})) = 0 
    \]
    which, after substituting into \eqref{originaleqationforxi}, can be solved for $t$. Substituting this $t$ back into \eqref{originaleqationforxi} yields \eqref{equationforxi}. \newline
    In \eqref{equationfort}, we need the term inside the $\log$ to be strictly positive, i.e.,
    \[
        0 < \sin^{2}(\lambda \phi(x_{0})) - \frac{\cos^{2}(\lambda \phi(x_{0}))}{\phi'(x_{0})^{2}}.
    \]
    This tells us $\lambda > \frac{1}{\phi(x_{0})}\tan^{-1}(\frac{-1}{\phi'(x_{0})})$. Similarly, in the expression for $\xi$ \eqref{equationforxi}, we need the term in the parenthesis to be greater than one, which is always true provided $\lambda < \frac{\pi}{2\phi(x_{0})}$.
\end{proof}

\begin{defi}\label{defA}
    Given $x_{0}$ as in the beginning of the section, and $\lambda$, $A_{t}^{\lambda,\xi}$ as in Lemma~\ref{Mainappendixlemma}, define $A^{\lambda}_{x_0}$ to the be connected component of $A_{t}^{\lambda,\xi}\cap \Omega$ passing through $(x_0, \phi(x_0))$.
\end{defi}

\subsection{Orthogonality at the Second Point}

For each $(x_{0},\phi(x_{0})) \in \partial\Omega$ as in the beginning of section A.1, we have one parameter family of Angenant ovals $A^{\lambda}_{x_{0}}$ which intersect $\partial\Omega$ orthogonally at $(x_{0},\phi(x_{0}))$. Now, let $\rho > 0$. We will now show that we can find $x_0$ and  $\lambda$ so that $A^{\lambda}_{x_{0}}$ intersects $\partial\Omega$ orthogonally at two points and lies in the strip $\{(x,y) \mid 0 < y < \rho\}$. We first show a preliminary lemma.

\begin{lem}\label{A3}
    Let $\rho$ be such that the line $y = \rho$ intersects $\partial\Omega$ at two points. Then there exists a point $x_{0}$ with $\phi(x_{0}) < \rho$ and $\phi'(x_{0}) < 0$, and such that the following holds:  
    $A^{\lambda}_{x_{0}}$ with  $\lambda := \frac{\pi}{2\rho}$, as in \ref{defA},  intersects $\partial\Omega$ at $(x_{0},\phi(x_{0}))$ orthogonally, and further, intersects $\partial\Omega$ at a second point $(\hat{x},\rho)$, where $\phi'(\hat{x}) > 0$.
\end{lem}
\begin{proof}
    By assumption, the line $y = \rho =\frac{\pi}{2\lambda}$ intersects $\partial\Omega$ at a point $(x_{1},\rho)$ where $\phi'(x_{1}) < 0$. Then for $x_{0} > x_{1}$, $\phi(x_{0}) = \rho$ and so $\frac{\pi}{2\phi(x_{0})} > \frac{\pi}{2\rho}$ from which the previous lemma implies $A^{\lambda}_{x_{0}}$ is well-defined. Recall that $A^{\lambda}_{x_{0}}$ is part of a scaled Angenent oval (as in Definition \ref{defA}) and the point on this Angenent oval with outward pointing unit normal equal to $-e_1$ has $y$-coordinate equal to $\rho$ and $x$-coordinate given by 
    \[
        \hat{x} = \frac{-1}{\lambda}\cosh^{-1}\left(\frac{1}{\sin^{2}(\lambda\phi(x_{0})) - \frac{\cos^{2}(\lambda\phi(x_{0}))}{\phi'(x_{0})}} \right) + \xi.
    \]
    As $x_{0} \to 1$, $\hat{x} \to -\infty$.
    Additionally, one can check (by replacing $\xi$ from Lemma~\ref{Mainappendixlemma}), as $x_{0} \searrow x_{1}$, $\hat{x} \to x_{0}$. Therefore, by the intermediate value theorem, we can find an $x_{0}$ satisfying the consequent of the lemma.
\end{proof}
\begin{lem}\label{defend}
    For any $\rho > 0$,  we can find $x_0$ and  $\lambda_\rho$, as in Lemma \ref{Mainappendixlemma} so that $\phi(x_0)<\rho$ and $A^{\lambda_\rho}_{x_{0}}$ intersects $\partial\Omega$ orthogonally at two points and lies in the strip $\{(x,y) \mid 0 < y < \rho\}$. 
\end{lem}
\begin{proof}
Given $\rho$, fix $x_{0}$ as given in Lemma \ref{A3}. Then, the angle between the tangent vectors of $A^{\frac{\pi}{2\rho}}_{x_0}$ and $\partial\Omega$ at $(\hat{x},\phi(\hat{x}))$ is less than $\pi/2$. Keeping $x_0$ fixed, we consider now $A^\lambda_{x_0}$ as in Definition \ref{defA} and we claim that there exists a $\lambda_\rho$ such that $A^{\lambda_\rho}_{x_0}$ intersects $\partial \Omega$ orthogonally (at both points). By continuity, it suffices to show that there exists a $\lambda$ such that the angle between the tangent vectors of $A^{\lambda}_{x_0}$ and $\partial\Omega$ at the \emph{second} point of intersection is less than $\pi/2$. We do this by showing, in the following claim, that there exists $\lambda$ for which the corresponding $\xi$ (determined in the definition for $A^\lambda_{x_0}$ in Lemma \ref{Mainappendixlemma}) is equal to $-1$. 
\begin{claim}
    There exists $\lambda \in \left( \frac{1}{\phi(x_{0})}\tan^{-1}(\frac{-1}{\phi'(x_{0})}), \frac{\pi}{2\phi(x_{0})} \right)$ such that $\xi = -1$.
\end{claim}
\begin{proof}
    $\xi = -1$ means we are trying to solve 
    \[
        x_{0} - \frac{1}{\lambda}\cosh^{-1}\left( \frac{\sin(\lambda \phi(x_{0}))}{\sqrt{\sin^{2}(\lambda \phi(x_{0})) - \frac{\cos^{2}(\lambda \phi(x_{0}))}{\phi'(x_{0})^{2}}}} \right) = -1,
    \]
    which after rearranging becomes 
    \[
        \phi'(x_{0})^{2}\tanh^{2}{(\lambda(x_{0}+1)) = \cot^{2}(\lambda\phi(x_{0}))}.
    \]
    Define a function
    \[
        f(\lambda) := \phi'(x_{0})^{2}\tanh^{2}{(\lambda(x_{0}+1)) - \cot^{2}(\lambda\phi(x_{0}))}.
    \]
    If $\lambda = \frac{1}{\phi(x_{0})}\tan^{-1}(\frac{-1}{\phi'(x_{0})})$, then
   \[
        f(\lambda) = \phi'(x_{0})^{2}\left( \tanh^{2}(\lambda(x_{0}+1)) - 1 \right) < 0.
   \]
   On the other hand, if $\lambda = \frac{\pi}{2\phi(x_{0})}$
   \[
        f(\lambda) = \phi'(x_{0})^{2}\tanh^{2}(\frac{\pi}{2\phi(x_{0})}(x_{0}+1)) > 0,
   \]
   this proves the claim.
\end{proof}
\end{proof}
\subsection{Calculating the asymptotic behavior of $\lambda_\rho$ (as in Lemma \ref{defend}) as $\rho\to0$}

Now that we have justified the construction of the shifted and scaled Angenent oval which intersects $\partial\Omega$ orthogonally below the horizontal line $y=\rho$ (Lemma \ref{defend}), we will show that its scale factor $\lambda_\rho$, as $\rho\to 0$, does have a limit, which we will call $\lambda_0$. To do this, we consider a sequence $\rho_{i}$ tending to $0$, and for each $\rho_{i}$ we considered the corresponding $A^{\lambda_{i}}_{x_{i}}$ as constructed in Lemma~\ref{defend}. We first show that $\liminf_{i \to \infty} \lambda_{i}$ and $\limsup_{i \to \infty} \lambda_{i}$  are bounded below and above respectively.

\begin{lem}\label{lambda0bounded}
	We have
    \[
    \max\{\k^{\Omega}(e_1),\k^{\Omega}(-e_1)\} \leq \liminf_{i \to \infty}\lambda_{i} \leq \limsup_{i \to \infty} \lambda_{i} \leq \sigma,
    \]
    where $\sigma$ solves the equation $\sigma\tanh{\sigma} = \max\{\k^{\Omega}(e_1),\k^{\Omega}(-e_1)\}$.
\end{lem}
\begin{proof}
    Consider a sequence $\rho_{i}\downarrow 0$ and the corresponding $A^{\lambda_{i}}_{x_{i}}$ as constructed in Lemma~\ref{defend}. Recall that 
    \[
        \lambda_{i} > \frac{1}{\phi(x_{i})}\tan^{-1}(\frac{-1}{\phi'(x_{i})}).
    \]
    Note that $x_{i} \to 1$, as $i\to \infty$, which gives $\liminf_{i \to \infty}\lambda_{i} > \k^{\Omega}(e_1)$, and since we could have done the entire construction in the previous section by considering $x_{i}$ so that $\phi'(x_{i}) > 0$, the lower bound follows.
We also note that, the above freedom to choose ``side'' for $x_i$, allows us to assume without loss of generality that $\kappa(e_1)\ge \kappa(-e_1)$ and, moreover, in the case of equality the following picture holds: If we denote by $R$ the reflection about the $y$-axis, then $R(\partial\Omega \cap \{x \geq 0, 0 \leq y \leq \rho_{i}\})$ lies inside $\overline\Omega$ for all $\rho_{i}$ sufficiently close to 0.

Consider now an Angenent oval $A_{x_{i}}^{\sigma_{i}}$ as in Lemma \ref{Mainappendixlemma} (see also Definition \ref{defA}) which is not shifted, that is
    \[
        \xi_{i} = x_{i} - \frac{1}{\sigma_{i}}\cosh^{-1}\left(e^{-\sigma_{i}^{2}t}\sin(\sigma_{i} \phi(x_{i})) \right) = 0. \tag{4}\label{xi}
    \]
    Then, $A_{x_{i}}^{\sigma_{i}}$ intersects $R(\partial\Omega \cap \{x \geq 0, 0 \leq y \leq \rho_{i}\})$ orthogonally, and thus $\partial\Omega \cap \{x < 0\}$ at an acute angle (as per defintion~\ref{definitionofacuteangle}, see figure~\ref{fig:reflectedangenant} and Remark~\ref{acuteangleremark}). Following the construction in Lemma~\ref{defend}, we point out that as the shift $\xi$ decreases, the scale $\lambda$ also decreases, and so by decreasing $\xi$ towards $-1$ (when the angle becomes obtuse), we can infer that $\lambda_{i} < \sigma_{i}$. Therefore $\limsup_{i \to \infty}\lambda_{i} \leq \sigma := \lim_{i \to \infty} \sigma_{i}$. To calculate $\sigma$, note that \eqref{xi} can be written as
    \[
        \cosh(\sigma_{i} x_{i}) = \frac{1}{\sqrt{1-\frac{\cot^{2}(\sigma_{i}\phi(x_{i}))}{\phi'(x_{i})^{2}}}}.
    \]
    After taking the limit as $x_{i} \to 1$, this becomes
    \[
        \cosh(\sigma) = \frac{1}{\sqrt{1-\frac{\k^{\Omega}(e_1)^{2}}{\sigma^{2}}}},
    \]
    which rearranges to 
    \[
        \sigma\tanh{\sigma} = \k^{\Omega}(e_1).
    \]
\end{proof}

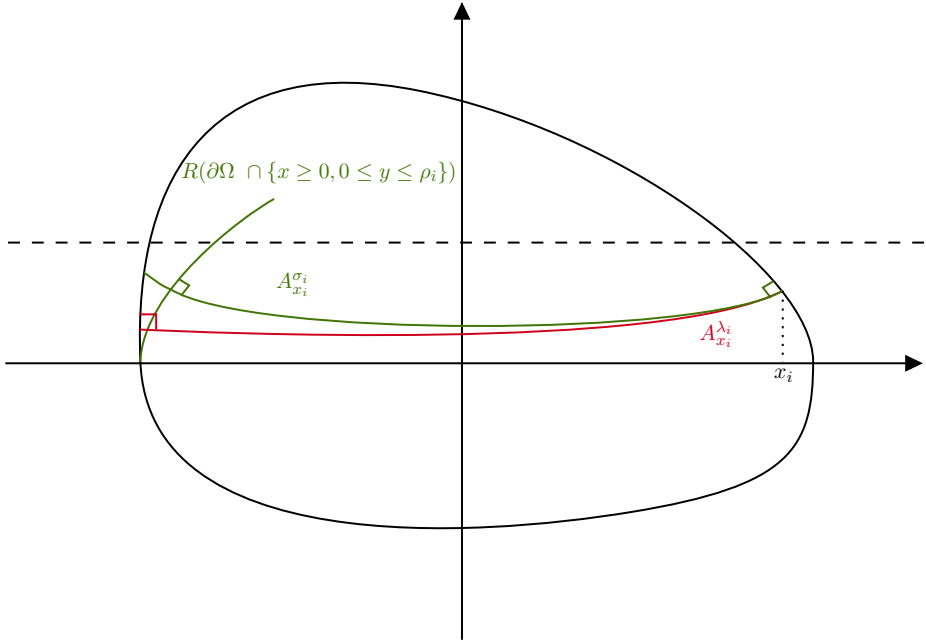
\begin{figure}[t]
\begin{center}
\tikzset{every picture/.style={line width=0.75pt}} 
\begin{tikzpicture}[x=0.75pt,y=0.75pt,yscale=-0.9,xscale=0.9]
\draw   (170.15,215.06) .. controls (165.63,104.91) and (214.64,40.54) .. (326.24,63.1) .. controls (437.83,85.67) and (547.17,165.96) .. (547.17,215.06) .. controls (547.17,264.17) and (534.35,284.74) .. (445.37,299.33) .. controls (356.4,313.93) and (174.68,325.21) .. (170.15,215.06) -- cycle ;
\draw    (94.5,215.95) -- (604.99,215.95) ;
\draw [shift={(607.99,215.95)}, rotate = 180] [fill={rgb, 255:red, 0; green, 0; blue, 0 }  ][line width=0.08]  [draw opacity=0] (8.93,-4.29) -- (0,0) -- (8.93,4.29) -- cycle    ;
\draw    (350.37,17) -- (350.37,371) ;
\draw [shift={(350.37,14)}, rotate = 90] [fill={rgb, 255:red, 0; green, 0; blue, 0 }  ][line width=0.08]  [draw opacity=0] (8.93,-4.29) -- (0,0) -- (8.93,4.29) -- cycle    ;
\draw  [dash pattern={on 4.5pt off 4.5pt}]  (96.01,148.26) -- (609.5,148.26) ;
\draw [color={rgb, 255:red, 65; green, 117; blue, 5 }  ,draw opacity=1 ]   (170.15,215.06) .. controls (171.51,175.25) and (227.01,133.84) .. (245.1,123.75) ;
\draw  [color={rgb, 255:red, 65; green, 117; blue, 5 }  ,draw opacity=1 ] (191.62,168.67) -- (197.5,172.25) -- (192.97,178.01) ;
\draw  [color={rgb, 255:red, 65; green, 117; blue, 5 }  ,draw opacity=1 ] (523.5,179.26) -- (519,173.5) -- (524.95,169.9) ;
\draw [color={rgb, 255:red, 208; green, 2; blue, 27 }  ,draw opacity=1 ]   (169.7,197.01) .. controls (204.69,198.6) and (463.47,210.28) .. (530.13,175.25) ;
\draw [color={rgb, 255:red, 65; green, 117; blue, 5 }  ,draw opacity=1 ]   (172.11,165.16) .. controls (176.3,167.5) and (180.94,172.36) .. (187.09,174.43) .. controls (230.02,201.79) and (485.79,201.79) .. (530.13,175.25) ;
\draw  [dash pattern={on 0.84pt off 2.51pt}]  (530.13,175.25) -- (530.13,215.06) ;
\draw  [color={rgb, 255:red, 208; green, 2; blue, 27 }  ,draw opacity=1 ] (169.7,188.5) -- (179,188.5) -- (179,197.01) ;
\draw (244.95,164.3) node [anchor=north west][inner sep=0.75pt]  [font=\small,xscale=0.75,yscale=0.75]  {${\textstyle \textcolor[rgb]{0.25,0.46,0.02}{A}\textcolor[rgb]{0.25,0.46,0.02}{_{x_{i}}^{\sigma _{i}}}}$};
\draw (482.31,190.84) node [anchor=north west][inner sep=0.75pt]  [font=\small,xscale=0.75,yscale=0.75]  {$\textcolor[rgb]{0.82,0.01,0.11}{A}\textcolor[rgb]{0.82,0.01,0.11}{_{x_{i}}^{\lambda _{i}}}$};
\draw (524.14,217.93) node [anchor=north west][inner sep=0.75pt]  [font=\small,xscale=0.75,yscale=0.75]  {$x_{i}$};
\draw (192,101.4) node [anchor=north west][inner sep=0.75pt]  [font=\small,color={rgb, 255:red, 65; green, 117; blue, 5 }  ,opacity=1 ,xscale=0.75,yscale=0.75]  {$R( \partial \Omega \ \cap \{x\geq 0,0\leq y\leq \rho_{i} \})$};
\end{tikzpicture}
\end{center}
\caption{$A^{\sigma_{i}}_{x_{i}}$ intersects $\partial\Omega$ at an acute angle.}
\label{fig:reflectedangenant}
\end{figure}

Lemma~\ref{lambda0bounded} implies there exists a subsequence $\rho_{i_{j}}$, for which the corresponding $\lambda_{i_{j}}$ converges, as $j \to \infty$, to some $\lambda_{0} \in (0,\infty)$. We will show that $\lambda_{0}$ is independent of the sequence $\rho_i$. This then implies that the scale $\lambda$ has a limit as $\rho \to 0$.

First note that the shifts $\xi_{i_{j}}$ of $A_{x_{i_j}}^{\lambda_{i_{j}}}$, as in \eqref{equationforxi} in Lemma~\ref{Mainappendixlemma}, also have a limit, which we call $\xi_0$ and satisfies 
\[
    \xi_{0} := 1 - \frac{1}{\lambda_{0}}\cosh^{-1}\left({\frac{1}{\sqrt{1-\frac{\k^{\Omega}(e_1)^{2}}{\lambda_{0}^{2}}}}}\right).
\]
Recall that if $A^{\lambda_{i_{j}}}_{x_{i_{j}}}$ intersects $\partial\Omega$ at a point $(\hat{x},\phi(\hat{x})) \in \partial\Omega$ orthogonally (here $\hat{x}$ is either $x_{i_{j}}$ or the corresponding $x$-coordinate of the point on $\partial\Omega \cap A^{\lambda_{i_{j}}}_{x_{i_{j}}}$, $(\hat{x},\phi(\hat{x}))$ with $\phi'(\hat{x}) > 0$), then
\[
    \tanh(\lambda_{i_{j}}(\hat{x} - \xi_{i_{j}}))\phi'(\hat{x}) + \cot(\lambda_{i_{j}} \phi(\hat{x})) = 0, 
\]
which can be written as
\[
    \tanh(\lambda_{i_{j}}(\hat{x} - \xi_{i_{j}})) = -\frac{\cot(\lambda_{i_{j}} \phi(\hat{x}))}{\phi'(\hat{x})}.
\]
Taking the limit as $j\to \infty$ and noting that $\hat{x} \to 1$ or  $\hat{x} \to -1$ we obtain
\[
\begin{cases}
\tanh(\lambda_{0}(1-\xi_{0})) &= \k^{\Omega}(e_1) \\
\tanh(\lambda_{0}(-1-\xi_{0})) &=  -\k^{\Omega}(-e_1).
\end{cases}\tag{5}\label{eq: lambda0system}
\]
The system \eqref{eq: lambda0system} can be reduced (by using the addition formula for the hyperbolic tangent) to 
\[
\lambda_{0}^{2} - \lambda_{0}(\k^{\Omega}(e_1) + \k^{\Omega}(-e_1))\coth{2\lambda_{0}} + \k^{\Omega}(e_1)\k^{\Omega}(-e_1) = 0,
\]
which, since $\lambda_{0} \geq \k^{\Omega}(e_1),\k^{\Omega}(-e_1)$, determines $\lambda_{0}$ uniquely.

\printbibliography

\end{appendix}

\end{document}